\documentclass{amsart}
\usepackage{graphicx}
\usepackage{amsmath}
\usepackage{mathrsfs}
\usepackage{amsmath, amscd}
\usepackage{amssymb}
\usepackage{tikz}
\usepackage{tikz-cd}
\usepackage{enumitem}

\newtheorem{theorem}{Theorem}[section]
\newtheorem*{theorem*}{Theorem}
\newtheorem*{convention*}{Convention}
\newtheorem{lem}[theorem]{Lemma}
\newtheorem{prop}[theorem]{Proposition}
\newtheorem{cor}[theorem]{Corollary}

\newtheorem{mainthm}{Theorem}

\theoremstyle{definition}
\newtheorem{defn}[theorem]{Definition}

\theoremstyle{remark}
\newtheorem{remark}[theorem]{Remark}

\numberwithin{equation}{section}

\newtheoremstyle{TheoremNum}
        {\topsep}{\topsep}              %%% space between body and thm
        {\itshape}                      %%% Thm body font
        {}                              %%% Indent amount (empty = no indent)
        {\bfseries}                     %%% Thm head font
        {.}                             %%% Punctuation after thm head
        { }                             %%% Space after thm head
        {\thmname{#1}\thmnote{ \bfseries #3}}%%% Thm head spec
    \theoremstyle{TheoremNum}
    \newtheorem{thmrep}{Theorem}

\newcommand{\pu}[1]{PU(p^{#1})}
\newcommand{\spl}[1]{Sp_{2{#1}}(\F_p)}
\newcommand{\splq}[1]{Sp_{2{#1}}(\F_q)}
\newcommand{\slq}[1]{SL_{{#1}}(\F_q)}
\newcommand{\glq}[1]{GL_{{#1}}(\F_q)}

\newcommand{\C}{\mathbb{C}}
\newcommand{\F}{\mathbb{F}}
\newcommand{\Z}{\mathbb{Z}}

\newcommand{\Gq}[1]{GL_{#1}(\mathbb{F}_q)}

\newcommand{\opn}{\operatorname}

\newcommand{\op}{\operatorname{P}}
\newcommand{\al}{\alpha}
\newcommand{\sG}{\mathscr{G}}
\newcommand{\sN}{\mathscr{N}}
\newcommand{\hbpu}{ H^*(B\pu{m};\F_p)}

\newcommand{\gmm}[1]{\gamma_{p^{2m}-p^{#1}}}
\newcommand{\pf}{\operatorname{pf}}
\newcommand{\td}{\tilde}

\newcommand{\sJ}{\mathscr{J}}
\newcommand{\sT}{\mathscr{T}}
\newcommand{\cH}{\mathcal{H}}
\newcommand{\cY}{\mathcal{Y}}
\newcommand{\br}{\mathfrak{b}}
\newcommand{\eg}{\mathfrak{e}_g}

\newcommand{\oin}{\iota}

\begin{document}

\title[$H^*(B\pu{m})$ and invariant polynomials]{The cohomology of $B\pu{m}$ and invariant polynomials}

\author{Xing Gu}
\address{Institute for Theoretical Sciences, School of Science, Westlake University, 600 Dunyu Road, Sandun town, Xihu district, Hangzhou 310030, Zhejiang Province, China.}
% \address{School of Science, Westlake University, 18 Shilongshan Road, Hangzhou 310024, Zhejiang Province, China.}
% \address{Institute of Natural Sciences, Westlake Institute for Advanced Study, 18 Shilongshan Road, Hangzhou 310024, Zhejiang Province, China.}

%\address{}
%\curraddr{}

\email{guxing@westlake.edu.cn}
\thanks{The author is supported by NSF Zhejiang, No.XHD24A0101.}

%\author{}
%\address{}
%\email{}

%    General info
%\subjclass[2000]{Primary 54C40, 14E20; Secondary 46E25, 20C20}
%\subjclass[2010]{55T10, 55R35, 55R40, 57T99}
%\date{June 14, 2016.}

	
%\date{}

%#\dedicatory{This paper is dedicated to our advisors.}


\begin{abstract}
	Let $p$ be an odd prime. For a compact Lie group $G$ and an elementary abelian $p$-group $A$ of $G$, one may define the Weyl group $W_A$ of $A$ in a similar fashion as defining the Weyl group of a maximal torus, such that $W_A$ acts on $H^*(BA;R)$ for any coefficient ring $R$, and the image of the restriction $H^*(BG;R)\to H^*(BA;R)$ lies in $H^*(BA;R)^{W_A}$, the sub-algebra of $H^*(BA:R)$ of $W_A$-invariant elements.  

	In this paper, we consider the projective unitary group $\pu{m}$ and one of its maximal elementary abelian $p$-subgroup $A_m$, of which the Weyl group is isomorphic to $\spl{m}$. Then the theory of $\spl{m}$-invariant polynomials over $\F_p$ may be applied to study the cohomology of $B\pu{m}$, the classifying space of $\pu{m}$. Following a theorem by Quillen, we deduce several theorems on $H^*(B\pu{m};\F_p)$ modulo the nilradical from results on invariant polynomials.

	% In particular, we describe the $\F_p$-submodule of $H^*(B\pu{m};\F_p)$ of mod $p$ reduction of torsion integral cohomology classes, modulo the nilradical, in terms of generators and relations.
	% Let $\sN$ be the nilradical of $H^*(B\pu{m};\F_p)$. The quotient ring $H^*(B\pu{m};\F_p)/\sN$ is a commutative $F_p$-algebra.  On can show that the restriction $H^*(B\pu{m};\F_p)\to H^*(B\pu{m};\F_p)$ factors through the quotient homomorphism  $H^*(B\pu{m};\F_p)\to H^*(B\pu{m};\F_p)/\sN$ and we obtain
\end{abstract}

\maketitle
%\tableofcontents
\section{Introduction}\label{sec:intro}
Let $n \geq 2$ be an integer. Let $PU(n)$ be the projective unitary group , i.e., the unitary group $U(n)$ modulo the central subgroup of scalar matrices. The cohomology of $BPU(n)$ is studied in various works including \cite{Kameko2008brown}, \cite{kono1975cohomology}, \cite{vavpetivc2005mod}, \cite{vezzosi2000chow}, \cite{vistoli2007cohomology}, \cite{gu2019cohomology}, \cite{gu2019some}, and \cite{fan2024cohomology}. However, apart from several isolated cases considered in the aforementioned works, our knowledge on the cohomology of $BPU(n)$ is limited. 

The cohomology of $BPU(n)$ plays a crucial role in the study of the topological period-index problem. For details see \cite{antieau2014topological}, \cite{antieau2014period} and \cite{gu2019topological}. Another interesting application of the cohomology of $BPU(n)$ is found in the study of the topological complexity of enumerative problems in algebraic geometry \cite{chen2024topological}.

In this paper, we fix a prime number $p > 2$ and an integer $m\geq 1$, and consider the cohomology algebra $\hbpu$. Let $\sN$ denote the nilradical of $\hbpu$, i.e., the ideal of nilpotent elements of $\hbpu$. Since cohomology algebras over $\F_p$ are graded commutative, the quotient $\hbpu/\sN$ is a graded commutative ring, necessarily concentrated in even dimensions. We employ Quillen's work \cite{quillen1971spectrum} to study this quotient.

There is a strong connection between $\hbpu/\sN$ and the subalgebras of the polynomial algebra $\F_p[x_1,y_1,\cdots,x_m,y_m]$ consisting of invariant polynomials under some $\spl{m}$-action and $GL_{2m}(\F_p)$-action. This is a key idea of this paper. 

We are able to determine a number of polynomial relations in the aforementioned quotient algebra. Moreover, we are able to acquire information on the integral cohomology of $B\pu{m}$. Let $\sT$ denote the ideal of $\hbpu$ generated by the mod $p$ reduction of torsion classes in $H^*(B\pu{m};\Z)$. We have a concrete description (Theorem~\ref{thm:ideals}) of $\sT$ modulo $\sN$. The ideal $\sT$ is a good approximation of the torsion subgroup of $H^*(B\pu{m};\Z)$, which is a $p$-primary torsion abelian group. 

It has been a long time since topologists noticed the connection between the cohomology of classifying spaces and polynomial invariants. This connection is discussed in \cite{adem2013cohomology} and \cite{wilkerson1983primer}, among other sources. 

Two collections of cohomology classes in $H^*(B\pu{m};\F_p)$ are essential in this paper, both of which are mod $p$ reductions of integral cohomology classes. 

The first collection of cohomology classes were studied in \cite{gu2019some}, where the author considers $p$-torsion integral cohomology classes $y_{p,k}\in H^{2(p^{k+1}+1)}(BPU(n);\Z)$. In this paper, we consider the mod $p$ reductions of $y_{p,i-1}$, denoted by $\alpha_i$. Indeed, for $p\mid n$, there is a generator $\bar\br$ of $H^3(BPU(n),\F_p)\cong\F_p$ such that 
\begin{equation}\label{eq:Steenrod}
	\alpha_i = \beta\op^{p^{i-1}}\op^{p^{i-2}}\cdots\op^1(\bar\br),
\end{equation}
where $\beta$ and $\op^k$ are the usual notations for Steenrod operations. 

The second collection of cohomology classes are the Chern classes of the conjugation representation, which is a complex representation of $PU(n)$ on $M(n)$, the complex vector space of $n\times n$ matrices. Let $\pi\colon U(n)\to PU(n)$ be the canonical projection. 
Let $\lambda\in PU(n)$, $\td\lambda\in U(n)$ satisfying $\lambda = \pi(\td\lambda)$, and let $\mu\in M(n)$. The conjugation action of $\lambda$ on $\mu$ is defined by $\lambda\circ\mu := \td\lambda\mu\td\lambda^{-1}$. 
This is a complex representation of $PU(n)$ of dimension $n^2$, or in other words, a homomorphism $PU(n)\to GL(n^2;\C)$, which is easily seen to be continuous. Therefore, it induces a map $BPU(n)\to BGL(n^2;\C)$. The pull-back of the universal Chern classes along this map are called the Chern classes of the conjugation  representation. The $i$th Chern classe of the conjugation representation of $PU(n)$ is denoted by $\gamma_i$.

The main theorems rely on Quillen's work on the spectrum of the cohomology ring of classifying spaces \cite{quillen1971spectrum}, which, modulo the nilradical, reduces the study of $H^*(BG;\F_p)$ for a compact Lie group $G$ to the study of the cohomology of the $p$-elementary abelian subgroups of $G$.
\begin{mainthm}\label{thm:ChernClasses}
	The Chern class $\gamma_i$ is nilpotent if 
	\begin{enumerate}[label = (\alph*)]
		\item $i$ is odd, or
		\item $i > p^{2m} - p^m$ and $i\neq p^{2m} - p^k$ for any positive integer $k$.
	\end{enumerate} 
\end{mainthm}

In Section~\ref{sec:Dickson}, we define polynomials $P_{i,j}$ $R_{i,j}$ and $D_{j}$ which appear in the following theorem. 
\begin{mainthm}\label{thm:polynomial}
	The following relations hold in $\hbpu$.
	\begin{equation}\label{eq:relation1}
		\sum_j^{i-1}\al_{i-j}^{p^j}\gmm{j} \equiv \sum_{j=i+1}^{2m}\al_{j-i}^{p^i}\gmm{j} \pmod{\sN}, \ 0\leq i\leq m-1,		
	\end{equation}

	\begin{equation}\label{eq:relation2}
		\begin{split}
			\gmm{i} \equiv
			\sum_{j=m-i}^m (-1)^{m+i+j}P_{m-i,j}(\al_1,\cdots,\al_{2j-1})^{p^{m-j}}&\gmm{m+j}
			\pmod{\sN}, \\
			& 0\leq i\leq m-1,
		\end{split}
	\end{equation}

	\begin{equation}\label{eq:relation4}
		\al_i \equiv \sum_{j=1}^{2m-1}(-1)^{j+1} R_{i,j}(\gmm{0},\cdots,\gmm{2m-1})\al_j\pmod{\sN},\ i\geq 2m,
	\end{equation}

	\begin{equation}\label{eq:relation3}
		D_i(\al_1,\cdots,\al_{2m})\equiv \gmm{i}  D_{2m}(\al_1\cdots,\al_{2m-1}) \pmod{\mathscr{N}},\ 0\leq i\leq 2m.
	\end{equation}
\end{mainthm}

\begin{remark}\label{rem:l=1}
	For $m = 1$, Vistoli \cite{vistoli2007cohomology} shows that all the mod $\sN$ equivalence relations in Theorem~\ref{thm:polynomial} are strict equations. For instance, corresponding to \eqref{eq:relation3}, we have $\al_2 = \gamma_{p^2-p}\al_1$.
\end{remark}
\begin{remark}\label{rem:l=2}
	For readers interested in concrete equations, we note that for $m = 2$, \eqref{eq:relation3} is as follows:
	\begin{equation}\label{eq:l=2}
		\begin{split}
		& \alpha_1^p \alpha_4  - \alpha_2 \alpha_3^p + \alpha_1 \alpha_2^{p^2}\equiv -\gamma_{p^4-p^3} (\alpha_1^{p^2+1} - \alpha_2^{p+1} + \alpha_1^p \alpha_3) \pmod{\mathscr{N}} ,\\
		& \alpha_1^{p^3+1} - \alpha_3^{p+1} + \alpha_2^p\alpha_4 \equiv \gamma_{p^4-p^2} (\alpha_1^{p^2+1} - \alpha_2^{p+1} + \alpha_1^p \alpha_3) \pmod{\mathscr{N}},\\
		& \alpha_1^{p^3}\alpha_2 - \alpha_2^{p^2}\alpha_3 + \alpha_1^{p^2}\alpha_4 \equiv - \gamma_{p^4-p} (\alpha_1^{p^2+1} - \alpha_2^{p+1} + \alpha_1^p \alpha_3) \pmod{\mathscr{N}},\\
		& \alpha_1^{p^3+p} - \alpha_2^{p^2+p} + \alpha_1^{p^2}\alpha_3^p \equiv \gamma_{p^4-1}(\alpha_1^{p^2+1} - \alpha_2^{p+1} + \alpha_1^p \alpha_3) \pmod{\mathscr{N}}.
	\end{split}
\end{equation}
\end{remark}	
% \begin{remark}\label{rem:algIndependence}
% 	Proposition \ref{pro:algIndependence} states that $\alpha_i$ for $1\leq i \leq4$ are algebraically independent. Therefore, the relations in Theorem \ref{thm:polynomial} are nontrivial.
% \end{remark}
\begin{defn}\label{def:subalg_G}
	The $\F_p$-subalgebra $\sG$ of $\hbpu$ is the subalgebra generated by $\al_i$ for $1\leq i\leq 2m-1$ and $\gmm{i}$ for $m\leq i\leq 2m$. 
	The ideal $\sJ$ of $\sG$ is the ideal generated by $\al_i$, $1\leq i\leq 2m-1$.
\end{defn}

\begin{remark}\label{rem:subalg_G}
	The algebra $\sG$ is unital since it contains $\gamma_0 = 1$. Let $\sG'$ be the $\F_p$-subalgebra of $\hbpu$ generated by $\al_i$ for all $i\geq 1$ and $\gamma_i$ for all $0\leq i\leq p^{2m}$. Let $\sJ'$ be the ideal of $\sG'$ generated by $\al_i$ for $i\geq 1$.
	As shown by Theorem~\ref{thm:polynomial}, we have 
	\[\sG\equiv\sG'\pmod{\sN} \textrm{ and } \sJ\equiv\sJ'\pmod{\sN}.\]
\end{remark}

% \begin{mainthm}\label{thm:torsion}
% 	Let $\varphi(\underline{\al})$ be a polynomial in the classes $\alpha_i$ for $i > 0$ such that $\varphi(\underline{\al})\in H^k(B\pu{m};\F_p)$ for some $k > 0$. Then we have the equality of $\F_p$-submodules of $\hbpu/\sN$ 
% 	\[\hbpu\varphi(\underline{\al})\equiv\sG\varphi(\underline{\al})\pmod{\sN}.\]
% \end{mainthm}
For the next theorem, recall that $\sT$ denotes the ideal of $\hbpu$ generated by the mod $p$ reduction of torsion classes in $H^*(B\pu{m};\Z)$.
\begin{mainthm}\label{thm:ideals}
	Regarding $\sT$, $\sJ$ and $\sN$ as $\F_p$-submodules of $\hbpu$, we have 
	\[\sT \equiv \sJ \pmod{\sN}.\]
\end{mainthm}

\begin{remark}\label{rem:reduction_to_A}
	There exists an elementary abelian $p$-subgroup $A_m$ of $\pu{m}$ with the following properties: the pull-backs of the equivalences of Theorem~\ref{thm:polynomial} and Theorem~\ref{thm:ideals} to $H^*(BA_m;\F_p)$ yield strict equations in which both sides are not $0$.
\end{remark}

Section~\ref{sec:Dickson} reviews the theory of invariant polynomials over a field $K$ of positive characteristic, although only the case $K = \F_p$ is needed for the rest of the paper. In Section~\ref{sec:p-subgroup}, we introduce an important elementary abelian $p$ -subgroup $A_m$ of $\pu{m}$ and its conjugation representation. In Section~\ref{sec:Chern}, the Chern classes of the conjugation action of $\pu{m}$,  restricted on $A_m$, are identified as the Dickson invariants, i.e., the $GL_{n}(\F_q)$-invariant polynomials. In Section~\ref{sec:ImOftheta} we interpret the purely algebraic facts introduced in Section~\ref{sec:Dickson} in terms of the cohomology of $BA_m$ and $B\pu{m}$. In Section~\ref{sec:Quillen}, we introduce a key argument by Quillen and prove Theorem~\ref{thm:ChernClasses} and Theorem~\ref{thm:polynomial}. In Section~\ref{sec:torsion}, we prove Theorem~\ref{thm:ideals}.

\subsection*{Acknowledgement} The author would like to thank the editor and the referee for their kind support. The author is grateful to the referee for a correction to the statement and an  elegant proof of Proposition~\ref{pro:dicksonAlternative}.

%%%%%%%%%%%%%%%%%%%%%%%%%%%%%%%%%%%%%%%%%%%%%%%%%%%%%%%%%%%%%%%%%%%%%%%%%%%%%%%%%%%%%%%%%%%%%%%%%%%%%%%%%%%%%%%%%

\section{Some invariant polynomials}\label{sec:Dickson}
Let $p$ be an odd prime and $q = p^r$ for some integer $r > 0$. Let $\F_q$ be the finite field of order $q$.  This section provides a brief overview of the theory of invariant polynomials over $\F_q$ under some actions of $\glq{n}$, $\slq{n}$, $\splq{m}$. The exposition follows \cite{benson1993polynomial} and  \cite{wilkerson1983primer},  with minor adjustment to notation. In addition, some new results are presented.

Let $V$ be a vector space over $\F_q$ of finite dimension $n$. Let $K$ be an extension filed of $\F_q$, and let $V_K = K\otimes_{\F_q}V$.
\begin{remark}\label{rem:K_Fq}
	For the arguments in later sections, the only case to be considered is that of $q = p$ and $K=\F_p$, although in this section the full generality is preserved, for consistency with the references.	
\end{remark}

Let $V^{\vee}$ be the dual vector space of $V$, and let $V^{\vee}_K$ be the $K$-dual of $V_K$. Let $K[V]$ be the $K$-algebra of polynomial functions on on $V_K$ taking values in $K$. The $K$-vector space of degree-$1$ polynomials in $K[V]$ is precisely $V^{\vee}_K$.

The group $\glq{n}$ has the left action $V^{\vee}$ dual to the tautological right action on $V$. This action extends uniquely to an action on $K[V]$ compatible with the $K$-algebra structure, which is the $\glq{n}$-action considered this section. The $\slq{n}$-action considered here is the restriction of the aforementioned $\glq{n}$-action. For $n = 2m$, a non-degenerate, antisymmetric bilinear form $\Omega$ is to be defined on $V$, so that the $\splq{m}$-action is to be defined as the restriction of the $\glq{n}$-action compatible with $\Omega$. 

We begin with the $\glq{n}$-invariant and the $\slq{n}$-invariant polynomials. Let $x'_1\cdots,x'_n$ be a basis for $V^{\vee}$, and let $x_i\in V^{\vee}_K$ be the $K$-linear extension of $x_i'$. In what follows, we regard elements in $K[V][X]$ as polynomials over $K$ in variables $x_1,\cdots,x_n,X$.
Let 
\begin{equation*}
	\Delta_n(x_1,\cdots,x_n,X) = \opn{det}
	\begin{pmatrix}
		x_1      & \cdots x_n       & X      \\
		x_1^q    & \cdots x_n^q     & X^q    \\
		\vdots   & \vdots \vdots    & \vdots \\
		x_1^{q^n}& \cdots x_n^{q^n} & X^{q^n} 
	\end{pmatrix}
	\in K[V][X],
\end{equation*}

and let 
\begin{equation*}
	f_n(x_1,\cdots,x_n,X) = \prod_{x\in V^{\vee}}(X - x)\in K[V][X].
\end{equation*}
Let
		\begin{equation}\label{eq:def_E}
			E_{n,i}(x_1,\cdots,x_n):= \opn{det}
			\begin{pmatrix}
				x_1            & \cdots &x_n       \\
				x_1^q          & \cdots &x_n^q     \\
				\vdots         & \vdots &\vdots    \\
				\widehat{x_1^{q^i}} & \cdots &\widehat{x_n^{q^i}} \\
				\vdots         & \vdots & \vdots   \\
				x_1^{q^n}      & \cdots &x_n^{q^n}   
			\end{pmatrix}
		\end{equation}
where the hats indicate that the row $(x_1^{q^i},\cdots,x_n^{q^i})$ is absent. 

We write $\Delta_n(X)$, $f_n(X)$ and $E_{n,i}$ when $x_1,\cdots,x_n$ are clear from the context.

The main result of \cite{dickson1911fundamental} is restated in \cite{wilkerson1983primer} as the following
\begin{theorem}[Theorem~1.2, \cite{wilkerson1983primer}]\label{thm:Dickson}
	\begin{enumerate}[label = (\alph*)]
		\item The polynomial $f_n(X)$ is of the form
		\begin{equation}\label{eq:fExpanded}
			f_n(X) = X^{q^n} + \sum_{i=0}^{n-1}(-1)^{n-i}C_{n,i}(x_1,\cdots,x_n)X^{q^i}
		\end{equation}
		where $C_{n,i}(x_1,\cdots,x_n)\in \F_q[V]$.
		\item We have $E_{n,n}\mid E_{n,i}$ for all $0\leq i\leq n$ and $C_{n,i} = E_{n,i}/E_{n,n}$.

		\item The polynomials $C_{n,0},C_{n,1},\cdots,C_{n,n-1}$ are algebraically independent in $K[V]$, and we have
		\begin{equation*}
			K[V]^{\Gq{n}} = K[C_{n,0},C_{n,1},\cdots,C_{n,n-1}].
		\end{equation*} 
	\end{enumerate}
\end{theorem}
When $n$ is clear from the context, we write $C_i$ for $C_{n,i}$ and $E_i$ for $E_{n,i}$.

\begin{theorem}[\cite{benson1993polynomial}, Theorem 8.2.1]\label{thm:SLinvarint}
	Let notations be as above. We have
	\begin{equation*}
		K[V]^{\slq{n}} = K[E_n,C_{1},\cdots,C_{n-1}].
	\end{equation*}
\end{theorem}

We have the following alternative description of $E_n$ and $C_i$.
\begin{prop}\label{pro:deltan}
	As a polynomial in the variables $x_i$, $E_n$ factors as $E_n = \prod\omega$, where $\omega$ runs over linear factors of the form
	\[\omega = x_i + \sum_{j=i+1}^n t_jx_j,\ t_j\in\F_q.\]
\end{prop}
\begin{proof}
	Since the base field $K$ is of characteristic $p$, we may apply the ``Freshman's dream'', i.e.,  $(x+y)^p = x^p + y^p$, and compute
	\[\omega^{p^k} =  x_i^{p^k} + \sum_{j=i+1}^n t_jx_j^{p^k},\ t_j\in\F_q,\]
	From which it follows $\omega\mid E_n$. Comparing the degrees, we have $\prod\omega = t E_n$ for some $t\in\F_q^{\times}$. Comparing the coefficient of the term $\prod_i x_i^{q^{i-1}}$, we have $t = 1$.
\end{proof}
% \begin{prop}[\cite{benson1993polynomial}, Proposition 8.1.3]\label{pro:dicksonAlternative}
% 	The Dickson invariant $C_k$ is equal to the sum over all subspaces $V'\leq V$ of codimension $k$, of the product of the $q^n - q^k$ linear forms in $V^{\vee}$ which have nonzero value on $V'$:
% 	\[C_k = \sum_{\opn{dim}(V')=n-k}\prod_{\omega\mid_{V'}\neq 0}\omega.\]
% \end{prop}

\begin{prop}\label{pro:dicksonAlternative}
	Let $1\leq m \leq n$. Then we have
	\begin{equation*}
		C_{n,i}(x_1,\cdots,x_{n-m},0,\cdots,0) = 
		\begin{cases}
			0, \ 1\leq i < m,\\
			C_{n-m, i-m}(x_1,\cdots,x_{n-m})^{q^m},\  m \leq i\leq n.
		\end{cases}
	\end{equation*} 

\end{prop}
\begin{proof}
	Consider the homomorphism
	\[\pi\colon K[x_1,\cdots,x_n]\to K[x_1,\cdots,x_{n-m}]\]
	such that $\pi(x_i) = x_i$ for $1\leq i\leq n-m$ and $\pi(x_i) = 0$ otherwise. This extends to a homomorphism between polynomial rings in $X$:
	\[\pi\colon K[x_1,\cdots,x_n][X]\to K[x_1,\cdots,x_{n-m}][X].\]
	Let $U$ be the subspace of the $\F_q$-vector space $V^{\vee}$ generated by $x_1,\cdots,x_{n-m}$. Then $\pi$ restricts to a linear projection $V^{\vee}\to U$ whose kernel is the subspace generated by $x_{n-m+1},\cdots,x_n$, which has $q^m$ elements. Therefore the preimage of each $y\in U$ has $q^m$ elements.

	Write $f_n(X) = \prod_{x\in V^{\vee}}(X - x)$. The argument above shows 
	\[\pi(f_n(X)) = \prod_{x\in V^{\vee}}(X - \pi(x)) = \prod_{y\in U}(X - y)^{q^m} = f_{n-m}(X)^{q^m}.\]
	Since $\F_q$ is of characteristic $p$, the ``Freshman's dream'' applies, and we have
	\begin{equation*}
		\begin{split}
			& \pi(f_n(X)) = f_{n-m}(X)^{q^m}\\
			= & [X^{q^{n-m}} + \sum_{j=0}^{n-m-1}(-1)^{n-m-j}C_{n-m,j}X^{q^j}]^{q^m}\\
			= & X^{q^n} + \sum_{i=m}^{n-1}(-1)^{n-i}C_{n-m,i-m}^{q^m}X^{q^{i}}.
		\end{split}
	\end{equation*}
	Comparing this with the equation
	\[\pi(f_n(X)) = X^{q^n} + \sum_{i=0}^{n-1}(-1)^{n-i}\pi(C_{n,i})X^{q^i},\]
	the desired equations follow.
\end{proof}
We consider the $\splq{m}$-invariant polynomials. For $n = 2m$, let $u_i,v_i$, $1\leq i\leq m$ be a basis for $V$, and let $\Omega$ be the antisymmetric bilinear form on $V^{\vee}$, which extends to one on $V_K^{\vee}$, as follows:
\[\Omega(u_i, u_j) = \Omega(v_i,v_j) = 0,\ \Omega(u_i,v_j) = \delta_{ij}.\]
Consider the right $\splq{m}$-action on $V$ preserving $\Omega$. The dual left $\splq{m}$-action on $V^{\vee}$ extends uniquely to an action on $K[V]$, which is the $\splq{m}$-action considered for the rest of this paper.  

For $w= \sum_{i=1}^m (t_i u_i + s_i v_i)$ where $t_i,s_i\in K$,  we have the $k$-fold Frobenius of $w$, denoted by
\[w^{q^k} = \sum_{i=1}^m (t_i^{q^k}u_i + s_i^{q^k}v_i).\]  

Let $x_i,y_i$, $1\leq i\leq m$ be the basis for $V^{\vee}$ dual to $u_i, v_i$. Consider the $\splq{m}$-invariant polynomial
\begin{equation}\label{eq:def_H}
	H_{2m,i}(x_1,y_1,\cdots,x_,y_m) = \sum_{j=1}^m(x_j y_j^{q^i} - x_j^{q^i}y_j).
\end{equation}
We write $H_i$ instead of $H_{2m,i}$ when $m$ is clear from the context.
 Regarding $H_i$ as a polynomial function on $V$, for $w\in V$, $i < j$, we
 have 
\begin{equation}\label{eq:alphaB}
	\Omega(w^{q^i}, w^{q^j}) = 
	(H_{j-i}(x_1,y_1,\cdots,x_m,y_m)^{q^i})(w).
\end{equation}
For the rest of this section, we write $H_i$ for $H_i(x_1,y_1,\cdots,x_m,y_m)$.
\begin{lem}[\cite{benson1993polynomial}, Lemma 8.3.1]\label{lem:algIndependenceh}
	$H_1,\cdots,H_{2m}$ are algebraically independent, and $K[V]$ is a finite separable extension of $K(H_1,\cdots,H_{2m})$.
\end{lem}

For an antisymmetric matrix $\mu$, Let $\pf(\mu)$ denote the Pfaffian of $\mu$. (See Proposition~8.3.3 of \cite{benson1993polynomial} for a formula for the Pfaffian over fields of positive characteristic.) 
% Here, $S_k$ denotes the permutation group on the set $\{1,\cdots,k\}$, and the wreath product is regarded as the subgroup of $S_{2m}$ generated by 
% \[\{\sigma\in S_{2m}\mid \sigma(2i) - 1 = \sigma(2i-1),\ 1\leq i\leq m\}\ \textrm{and}\ \{(2i-1, 2i)\in S_{2m}\mid 1\leq i\leq m\}.\]  
Define polynomials 
\[D_{2m,i}(Y_1,\cdots,Y_{2m})\in K[Y_1,\cdots, Y_{2m}],\ 0 \leq i\leq 2m\]
as follows. Let $(G_{i,j})_{0\leq i,j\leq 2m}$ be the $(2m+1)\times (2m+1)$ skew-symmetric matrix with entries in the polynomial ring $K[Y_1,\cdots,Y_{2m}]$,  such that for $j > i$, $G_{i,j} = Y_{j-i}^{q^i}$. Then let  
\begin{equation}\label{eq:def_D}
	D_{2m,i}(Y_1,\cdots,Y_{2m}) = \pf
	\begin{pmatrix}
		G_{0,0} &\cdots & \widehat{G}_{0,i} &\cdots & G_{0,2m}\\
		\vdots & & \vdots & & \vdots\\
		\widehat{G}_{i,0}& \cdots &\widehat{G}_{i,i}& \cdots & \widehat{G}_{i,2m}\\
		\vdots & & \vdots & & \vdots\\
		G_{2m,0} &\cdots & \widehat{G}_{2m,i} &\cdots & G_{2m,2m}
	\end{pmatrix}.
\end{equation}
The hats indicate missing rows and columns. When $m$ is clear from the context, we write $D_i$ for $D_{2m,i}$. 
By \eqref{eq:alphaB}, we have
\begin{equation}\label{eq:Pfaffian}
	D_{2m,i}(H_1,\cdots,H_{2m}) = \pf
		\begin{pmatrix}
			\Omega(w,w)& \cdots& \widehat{\Omega}(w,w^{q^i}) &\cdots & \Omega(w,w^{q^{2m}})\\
			\vdots&        & \vdots            &       & \vdots  \\
			\widehat{\Omega}(w^{q^i},w) & \cdots &  \widehat{\Omega}(w^{q^i},w^{q^i}) & \cdots &\widehat{\Omega}(w^{q^i}, w^{q^{2m}})\\
			\vdots&    &  \vdots &    & \vdots\\
			\Omega(w^{q^{2m}},w)&\cdots & \widehat{\Omega}(w^{q^{2m}}, w^{q^i}) &\cdots &\Omega(w^{q^{2m}},w^{q^{2m}})
		\end{pmatrix}.
\end{equation}
By Proposition~8.3.3 of \cite{benson1993polynomial} we have
\begin{prop}\label{pro:PfaffianEqProduct}
	For $0 \leq i\leq 2m$, we have 
	\begin{equation}\label{eq:PfaffianEqProduct}
		\sum_i D_{i}(H_1,\cdots,H_{2m})X^{q^i} =E_{2m}(x_1,y_1,\cdots,x_m,y_m) \prod_{x\in V^{\vee}}(X - x).
	\end{equation}
	The polynomial $D_{2m}(Y_1,\cdots,Y_{2m})$ is independent of $Y_{2m}$, hence may be written as $D_{2m}(Y_1,\cdots,Y_{2m-1})$, and we have 
	\begin{equation}\label{eq:Dm0}
		D_{2m}(H_1,\cdots,H_{2m-1}) = E_{2m}(x_1,y_1,\cdots,x_m,y_m).
	\end{equation}
\end{prop}
\begin{proof}
	Equation~\eqref{eq:PfaffianEqProduct} is Proposition 8.3.3, \cite{benson1993polynomial}. It follows from \eqref{eq:def_D} and Lemma~\ref{lem:algIndependenceh} that $D_{2m}$ is independent of $Y_{2m}$. Equation~\eqref{eq:Dm0} following from a comparison of the leading coefficients of both sides of \eqref{eq:PfaffianEqProduct}.
\end{proof}

Comparing \eqref{eq:fExpanded}, \eqref{eq:PfaffianEqProduct} and \eqref{eq:Dm0}, we obtain
\begin{cor}\label{cor:CandD}
	For $0\leq i\leq 2m$, we have
	\begin{equation*}
			D_i(H_1,\cdots,H_{2m}) = (-1)^iC_i(x_1,y_1,\cdots,x_m,y_m) D_{2m}(H_1,\cdots,H_{2m-1}).
	\end{equation*}
	\qed
\end{cor}

\begin{prop}[\cite{benson1993polynomial}, Proposition 8.3.7]\label{pro:Pij}
	There are polynomials 
	\[P_{i,j}\in K[Y_1,\cdots,Y_{2j-1}],\ 1\leq i\leq j,\] 
	independent of $m$, such that for $1\leq i\leq m$, 
	\[D_{m-i}(H_1,\cdots,H_{2m}) = \sum_{j=i}^m P_{i,j}^{q^{m-j}}(H_1,\cdots,H_{2j-1})^{q^{m-j}}D_{m+j}(H_1,\cdots,H_{2m}).\] 
	As a polynomial in $x_1,y_1,\cdots,x_m,y_m$, $P_{i,j}(H_1,\cdots,H_{2j-1})$ is of degree $q^{2j} - q^{j-i}$.
	\qed
\end{prop}

The following theorem fully describes the ring $K[V]^{\splq{m}}$. 
\begin{theorem}[\cite{benson1993polynomial}, Theorem 8.3.11]\label{thm:sympInv}
	The ring $K[V]^{\splq{m}}$ is generated by the elements
	\[C_0,\cdots,C_{2m-1},H_1,\cdots,H_{2m}.\]
	subject only to the following relations:
	\begin{equation}\label{eq:sympInv1}
		\sum_{j=0}^{i-1}(-1)^jH_{i-j}^{q^j}C_j = \sum_{j=i+1}^{2m}(-1)^jH_{j-i}^{q^j}C_j,\ 0\leq i\leq m-1,
	\end{equation}
	\begin{equation}\label{eq:sympInv2}
		C_i = \sum_{j=m-i}^mP_{m-i,j}(H_1,\cdots,H_{2j-1})^{q^{m-j}}C_{m+j},\ 0\leq i\leq m-1.
	\end{equation}
\end{theorem}

\begin{cor}\label{cor:sympInv}
	The ring $K[V]^{\splq{m}}$ is generated by the elements
	\[C_m,\cdots,C_{2m-1},H_1,\cdots,H_{2m-1}.\]
\end{cor}
\begin{proof}
	This follows from \eqref{eq:sympInv1} with $i = 0$, and \eqref{eq:sympInv2}.
\end{proof}
To avoid ambiguity, we write $C_{2m,i}$ and $H_{2m,i}$ instead of $C_i$ and $H_i$ throughout the rest of this section.

Let $\cH$ denote the ideal of $K[V]$ generated by $H_{2m,1},\cdots,H_{2m,2m-1}$. Let 
\[\rho\colon K[V]\to K[V]/\cH\] 
be the quotient homomorphism.
\begin{lem}\label{lem:alg_indep}
	The elements $\rho(C_{2m,i})$ for $m\leq i\leq 2m-1$ are algebraically independent in $K[V]/\cH$.
\end{lem}
\begin{proof}	
	Let $\cY$ be the ideal of $K[V]$ generated by $y_1,\cdots,y_m$. We have $\cH\subseteq\cY$. Let $\pi\colon K[V]\to K[V]/\cY$ be the quotient homomorphism. It suffices to show that $\pi(C_{2m,i})$ for $m\leq i\leq 2m-1$ are algebraically independent in 
	\[K[V]/\cY\cong K[x_1,\cdots,x_m].\] 
	For $m\leq i\leq 2m-1$, we have
	\begin{equation*}
		\begin{split}
			\pi(C_{2m,i}) & = C_{2m,i}(x_1,0,x_2,0,\cdots,x_m,0) \\
			& = C_{2m,i}(x_1,x_2,\cdots,x_m,0,\cdots,0) \ (\textrm{since $C_{2m,i}$ is $\glq{m}$-invariant})\\
			& = C_{m,i-m}(x_1,x_2,\cdots,x_m)^{q^m}(\textrm{by Proposition~\ref{pro:dicksonAlternative}}).\\
		\end{split}	
	\end{equation*}
	By Theorem~\ref{thm:Dickson}, the elements $\pi(C_{2m,i})$ are algebraically independent, and the proof is completed.
\end{proof}
Let $\cH'$ be the ideal of $K[V]^{\splq{m}}$ generated by $H_{2m,1},\cdots,H_{2m,2m-1}$.
\begin{prop}\label{pro:idealH}
	In $K[V]$, we have 		$\cH' = K[V]^{\splq{m}}\cap\cH$.
	
\end{prop}
\begin{proof}
	The inclusion $\cH' \subseteq K[V]^{\splq{m}}\cap\cH$ is clear. It remains to show the inclusion of the other direction. Let $u\in K[V]^{\splq{m}}\cap\cH$. By Corollary~\ref{cor:sympInv}, we have a polynomial $F\in K[X_m,\cdots,X_{2m-1}]$ and $u'\in\cH'$ such that
	\[u = F(C_{2m,m},\cdots,C_{2m,2m-1}) + u'.\]
	By $u\in\cH$, we have 
	\[0 = \rho(u) = F(\rho(C_{2m,m}),\cdots,\rho(C_{2m,2m-1})).\]
	By Lemma~\ref{lem:alg_indep}, we have $F =0$, and $u = u'\in\cH'$, which concludes the proof.
\end{proof}

\begin{lem}[\cite{benson1993polynomial}, Lemma~8.3.8]\label{lem:OmegaDet}
	If $w_0,\cdots,w_{2m}$ are elements of $V$, then
	\[\sum_{j=0}^{2m}(-1)^j\Omega(w_0,w_j)\opn{det}(w_0,\cdots,w_{j-1},w_{j+1},\cdots,w_{2m}) = 0.\]
	\qed
\end{lem}
For the rest of this section, we write $C_i$ for $C_{2m,i}(x_1,y_1,\cdots,x_m,y_m)$, unless otherwise noted.
\begin{prop}\label{pro:h2m}
	There are polynomials $R_{2m,i,j}$ (or simply $R_{i,j}$ when $m$ is clear from the context) for $1\leq j\leq 2m-1$, such that for $i\geq 2m$, we have
	\begin{equation}\label{eq:hi}
		H_i = \sum_{j=1}^{2m-1}(-1)^{j+1} R_{i,j}((-1)^0C_0,\cdots,(-1)^{2m-1}C_{2m-1})H_j,\ i\geq 2m.
	\end{equation}
\end{prop}
\begin{proof}
	Let $i\geq 2m$. For $v\in V$, take $w_j$ in Lemma~\ref{lem:OmegaDet} to be $w_j = w^{q^j}$ for $j\leq 2m-1$ and $w_{2m} = w^{q^i}$. Then Lemma~\ref{lem:OmegaDet} yields 
	\begin{equation*}
		\begin{split}
			& H_i\opn{det}(w,\cdots,w^{q^{2m-1}}) \\ = & \sum_{j=0}^{2m-1}(-1)^{j+1} H_j\opn{det}(w,\cdots,w^{q^{j-1}},w^{q^{j+1}},\cdots, w^{q^{2m-1}},w^{q^i}).
		\end{split}
	\end{equation*}
	
	Since $H_0 = 0$, we obtain  
	\begin{equation*}
		\begin{split}
			& H_i\opn{det}(w,\cdots,w^{q^{2m-1}}) \\ = & \sum_{j=1}^{2m-1}(-1)^{j+1} H_j\opn{det}(w,\cdots,w^{q^{j-1}},w^{q^{j+1}},\cdots, w^{q^{2m-1}},w^{q^i}).
		\end{split}
	\end{equation*}

	The determinants on both sides of the equation may be regarded as functions sending $w$ to a value in $K$. Indeed, they are polynomials in the linear functions  $x_i,y_i$. The polynomial $\opn{det}(w,\cdots,w^{q^{j-1}},w^{q^{j+1}},\cdots,w^{q^i})$ is divisible by any polynomial in $x_i,y_i$ of degree $1$. By Proposition~\ref{pro:deltan}, we have 
	\begin{equation}\label{eq:detDivisible}
		\opn{det}(w,\cdots,w^{q^{2m-1}})\mid\opn{det}(w,\cdots,w^{q^{j-1}},w^{q^{j+1}},\cdots,w^{q^{2m-1}}, w^{q^i}),
	\end{equation}
	and a routine computation shows that the quotient is in $\F_q[V]^{\glq{2m}}$. By Proposition~\ref{thm:Dickson},
	this quotient is equal to $R_{2m,i,j}((-1)^0C_0,\cdots, (-1)^{2m-1}C_{2m-1})$ for some polynomial $R_{2m,i,j}$ in $2m$ variables, and the conclusion follows.
\end{proof}
% \begin{cor}\label{cor:C}
% 	Any nonzero polynomial in $C_i$, $m\leq i\leq 2m$ is not in the ideal generated by all $h_i$, $i\geq 1$.   
% \end{cor}

%%%%%%%%%%%%%%%%%%%%%%%%%%%%%%%%%%%%%%%%%%%%%%%%%%%%%%%%%%%%%%%%%%%%%%%%%%%%%%%%%%%%%%%%%%%%%%%%%%%%%%%%%%%%%%%%%%

\section{An elementary abelian $p$-subgroup of $\pu{m}$}\label{sec:p-subgroup}
For any integer $n>1$, let $M(n)$ denote the complex vector space of $n\times n$ matrices. Let $f\colon U(n)\to PU(n)$ be the quotient map.
\begin{defn}\label{def:conj}
	Let $\lambda\in PU(n)$ and let $\td{\lambda}\in f^{-1}(\lambda)$. The \textit{conjugation representation} of $PU(n)$, which we denote by $\Psi$, is defined by 
		\begin{equation*}
			\Psi: PU(n)\times M(n)\to M(n),\ (\lambda,\mu)\mapsto \lambda\circ \mu := \td{\lambda} \mu\td{\lambda}^{-1}.
		\end{equation*}
	The restriction of the conjugation representation to a subgroup of $PU(n)$ is also called a conjugation representation. 
\end{defn} 
\begin{remark}
	The more natural way to define $\Psi$ is by 
	\begin{equation*}
		\Psi: PU(n)\times M(n)\to M(n),\ (\lambda,\mu)\mapsto \lambda\circ \mu := \td{\lambda} \mu\td{\lambda}^\dagger.
	\end{equation*}
	where $\td{\lambda}^\dagger$ is the conjugate-transpose of $\lambda$. Since $\td{\lambda}$ is unitary, we have $\td{\lambda}^\dagger = \td{\lambda}^{-1}$. The choice made in Definition~\ref{def:conj} works better for computing characteristic classes.
\end{remark}

Let $(\C^n)^{\vee}$ denote the dual of $\C^n$. There is a canonical isomorphism $M(n)\cong \C^n\otimes(\C^n)^{\vee}$. For $\lambda\in PU(n)$, let $\td{\lambda}\in U(n)$ be a matrix representing $\lambda$, and let $v\in\C^n$, $\varphi\in(\C^n)^{\vee}$. The conjugation action of $PU(n)$ on $ M(n) $ is given by
	\begin{equation}\label{eq:conjp}
		\lambda\circ (v\otimes\varphi) = \td{\lambda}v\otimes \varphi \td{\lambda}^{-1},
	\end{equation}
The notations $\td{\lambda}v$ and $\varphi \td{\lambda}^{-1}$ denote the usual multiplication of matrices, with $v$ regarded as a column vector and $\varphi$ a row vector.

For the rest of this section, we fix an odd prime $p$. Let $V = \C^p$, and let $ V^{\vee}$ denote the dual of $V$. Let $ M :=  M(p)$. Then we have the canonical isomorphism $ M\cong V\otimes V^{\vee}$ as the complex vector space of $p\times p$ matrices. 

Let $e_1,\cdots,e_p$ be the canonical basis for $V$, and we order the elements 
\[e_{i_1}\otimes\cdots\otimes e_{i_m}\in V^{\otimes m},\ i=1,\cdots,p\]
lexicographically. Then they form an ordered basis for $V^{\otimes m}$, which identifies $V^{\otimes m}$ with $\C^{p^m}$. In a similar fashion we identify $(V^{\vee})^{\otimes m}$ with $(\C^{p^m})^{\vee}$.

We regard $M(p^m)$ as $V^{\otimes m}\otimes (V^{\vee})^{\otimes m}$, and further identify it as $ M^{\otimes m}$ via the ``shuffle'' isomorphism
	\begin{equation*}
		\begin{split}
			\opn{Sh}\colon (V\otimes V^{\vee})^{\otimes m} &\xrightarrow{\cong} V^{\otimes m}\otimes (V^{\vee})^{\otimes m}\cong \C^{p^m}\otimes(\C^{p^m})^{\vee},\\
			(v_1\otimes\varphi_1)\otimes\cdots\otimes(v_m\otimes\varphi_m) &\mapsto (v_1\otimes\cdots\otimes v_m)\otimes (\varphi_1\otimes\cdots\otimes\varphi_m). 
		\end{split}	
	\end{equation*}

In particular, for $\td{\lambda}_i\in U(p)$, $1\leq i\leq m$, $\td{\lambda}_1\otimes\cdots\otimes\td{\lambda}_m$ is identified as a matrix in $ M(p^m)$ via $\opn{Sh}$. The left multiplication of  $\td{\lambda}_1\otimes\cdots\otimes\td{\lambda}_m$ with a column vector in $\C^{p^m}\cong V^{\otimes m}$ is given by 
\[(\td{\lambda}_1\otimes\cdots\otimes\td{\lambda}_m)(v_1\otimes\cdots\otimes v_m) = \td{\lambda}_1v_1\otimes\cdots\otimes\td{\lambda}_mv_m,\]
and this identification is compatible with the multiplication of matrices, in the sense that we have
\[(\td{\lambda}_1\otimes\cdots\otimes\td{\lambda}_m)(\td{\lambda}_1'\otimes\cdots\otimes\td{\lambda}_m') = \td{\lambda}_1\td{\lambda}_1'\otimes\cdots\otimes\td{\lambda}_m\td{\lambda}_m'.\]
It is easily verified that $\td{\lambda}_1\otimes\cdots\otimes\td{\lambda}_m$ is indeed a member of $U(p^m)$. Let
 \[\lambda_1\otimes\cdots\otimes\lambda_m = f(\td{\lambda}_1\otimes\cdots\otimes\td{\lambda}_m).\] Again, the tensor product is compatible with the multiplication rules of $PU(p)$ and $\pu{m}$, in the sense that for $\lambda_i,\lambda_i'\in PU(p)$, we have 
\[(\lambda_1\otimes\cdots\otimes\lambda_m)(\lambda_1'\otimes\cdots\otimes\lambda_m') = \lambda_1\lambda_1'\otimes\cdots\otimes\lambda_m\lambda_m'.\]
A routine computation yields the following
\begin{lem}\label{lem:conjTensor}
	For $1\leq i\leq m$, let $\lambda_i\in PU(p)$, and let $\td{\lambda}_i\in U(p)$ representing $\lambda_i$. Let $v_i\in V$, $\varphi_i\in V^{\vee}$. Then we have 
	\[(\lambda_1\otimes\cdots\otimes\lambda_m)\circ (v_1\cdots v_m\otimes\varphi_1\otimes\cdots\otimes\varphi_m) = \opn{Sh}(\lambda_1\circ (v_1\otimes\varphi_1)\otimes\cdots\otimes\lambda_m\circ(v_m\otimes\varphi_m)).\] 
	\qed
\end{lem}
% By \eqref{eq:conjp}, we have 
% 		\begin{equation*}
% 			\begin{split}
% 				& (\lambda_1\otimes\cdots\otimes \lambda_m)\circ (v_1\otimes\cdots \otimes v_m\otimes\varphi_1\otimes\cdots\otimes\varphi_m) \\
% 				 = & [(\td{\lambda}_1\otimes\cdots\otimes\td{\lambda}_m)(v_1\otimes\cdots \otimes v_m)]\otimes [(\varphi_1\otimes\cdots\otimes\varphi_m)(\td{\lambda}_1\otimes\cdots\otimes\td{\lambda}_m)^{-1}]\\
% 				 = &\opn{Sh}((\td{\lambda}_1 v_1\otimes \varphi_1\td{\lambda}_1^{-1})\otimes\cdots\otimes(\td{\lambda}_m v_m\otimes \varphi_m\td{\lambda}_m^{-1}))\\
% 				 = &\opn{Sh}((\lambda_1\circ (v_1\otimes\varphi_1))\otimes\cdots\otimes(\lambda_m\circ(v_m\otimes\varphi_m))
% 			\end{split}
% 		\end{equation*}

We proceed to consider an elementary $p$-subgroups of $\pu{m}$. The special case $m=1$ is studied in details in \cite{kono1993brown} and \cite{vistoli2007cohomology}.

More generally, nontoral elementary abelian $p$-subgroups of $PU(n)$ are studied in \cite{griess1991elementary} %(Table \RomanNumeralCaps{2})
and \cite{andersen2008classification}. 

For $m =1$, let  
\begin{equation}\label{eq:sigmaHat}
	\omega := e^{2\pi i/p},\ 
	\tilde{\sigma} := 
	\begin{pmatrix}
		\omega & & & \\
		& \ddots & & \\
		& & \omega^{p-1} & \\
		& & & 1
	\end{pmatrix},\
	\tilde{\tau} :=\begin{pmatrix}
		& 1\\
		I_{p-1} &
	\end{pmatrix}.
\end{equation} 
Let $\sigma = f(\td{\sigma})$ and $\tau = f(\td{\tau})$. For $i,j = 0,\cdots,p - 1$, we define $\mu_{i,j}\in M$ as follows:
\begin{equation}\label{eq:Aij}
	\mu_{i,0} = 
	\begin{pmatrix}
		0 & I_i \\
		I_{p - i} & 0
	\end{pmatrix},\ 
	\mu_{0,j} = 
	\begin{pmatrix}
		\omega^{(p-1)j} & & & & \\
		& \omega^{(p-2)j} & & & \\
		& & \ddots & & \\
		& & & \omega^j & \\
		& & & & 1 \\
	\end{pmatrix},\ 
	\mu_{i,j} = \mu_{i,0}\mu_{0,j}.
\end{equation}
Notice that the formally different definitions of $\mu_{0,0}$ coincide with one another, since they all yield $\mu_{0,0} = I_p$. Let $M_{i,j}$ be the linear subspace of $M$ spanned by $\mu_{i,j}$.

Let $A_1$ be the subgroup of $PU(p)$ generated by $\sigma$ and $\tau$. With the notations in \eqref{eq:sigmaHat}, we have
\[\td{\sigma}\td{\tau} = \omega\td{\tau}\td{\sigma},\]
and therefore, we have $\sigma\tau = \tau\sigma$. It follows that we have 
$A_1\cong\Z/(p)\times\Z/(p)$.
Let $\iota\colon A_1\to PU(p)$ be the inclusion. The conjugate representation $\Psi$ of $PU(p)$ with ambient space $M$ has a restriction on $A_1$, which we denote by $\iota^*(\Psi)$. 

\begin{lem}\label{lem:split}
	As the ambient space of the representation $\oin^*(\Psi)$, $ M$ splits as follows:
	\begin{equation*}
		 M = \bigoplus_{0\leq i,j\leq p-1}M_{i,j}.
	\end{equation*}
	Each $M_{i,j}$ is an invariant subspace of $\oin^*(\Psi)$ satisfying
	\begin{equation}\label{eq:sigmaAij}
		\sigma\circ \mu_{i,j} = \omega^i \mu_{i,j},\ \tau\circ \mu_{i,j} = \omega^j \mu_{i,j}.
	\end{equation}	
\end{lem}

\begin{proof}
	A straightforward computation yields \eqref{eq:sigmaAij}, from which we deduce that $M_{i,j}$ are pairwise different invariant subspaces of $\oin^*(\Psi)$. Therefore, we have
	\begin{equation*}
		\bigoplus_{0\leq i,j\leq p-1}M_{i,j}\subseteq M.
    \end{equation*}
	Notice that both sides of ``$\subseteq$'' are of dimension $p^2$. Therefore, the two sides are equal and the proof is completed. 
\end{proof}

For $m \geq 1$, let 
\begin{equation}\label{eq:def_sigma}
	\sigma_i = 1\otimes\cdots 1\otimes\underbrace{\sigma}_{\text{$i$th tuple}}\otimes 1\cdots\otimes1\in\pu{m}.
\end{equation}
%where $\tilde{\sigma}$ occupies the $i$th entry, 
Similarly, let
\begin{equation}\label{eq:def_tau}
	\tau_i = 1\otimes\cdots 1\otimes\underbrace{\tau}_{\text{$i$th tuple}}\otimes 1\cdots\otimes1\in\pu{m}.
\end{equation}
%where $\tilde{\tau}$ occupies the $i$th entry, 
For $m\geq 1$, let 
	\begin{equation*}
		\mathbb{I} = \F_p^{2m} = \{(i_1,j_1,\cdots,i_m,j_m)\mid i_k,j_k\in \F_p\}.
	\end{equation*}
	For $(i_1,j_1,\cdots,i_m,j_m)\in\mathbb{I}$, let
	\begin{equation*}
		\mu_{i_1,j_1,\cdots,i_m,j_m} = \mu_{i_1,j_1}\otimes \mu_{i_2,j_2}\otimes\cdots\otimes \mu_{i_m,j_m}\in M^{\otimes m}.
	\end{equation*}
Let $M_{i_1,j_1,\cdots,i_m,j_m}$ be the subspace of $M^{\otimes m}$ generated by $\mu_{i_1,j_1,\cdots,i_m,j_m}$.

\begin{prop}\label{pro:split}
	For $(i_1,j_1,\cdots,i_m,j_m)\in\mathbb{I}$, we have
	\begin{equation}\label{eq:split_sigma}
		\begin{split}
			& \sigma_k \circ \mu_{i_1,j_1,\cdots,i_m,j_m} = \omega^{i_k}\mu_{i_1,j_1,\cdots,i_m,j_m},\\ & \tau_k \circ \mu_{i_1,j_1,\cdots,i_m,j_m} = 
			\omega^{j_k}\mu_{i_1,j_1,\cdots,i_m,j_m}.
		\end{split}
	\end{equation}	
	Furthermore, the vector space $ M^{\otimes m}$ splits as
	\begin{equation}\label{eq:split_M}
		 M^{\otimes m} = \bigoplus_{(i_1,j_1,\cdots,i_m,j_m)\in\mathbb{I}}M_{i_1,j_1,\cdots,i_m,j_m}.
	\end{equation}
\end{prop}

\begin{proof}
	By Lemma~\ref{lem:conjTensor}, we have 
	\[\sigma_k \circ \mu_{i_1,j_1,\cdots,i_m,j_m} = \mu_{i_1,j_1}\otimes\cdots\otimes(\sigma\circ \mu_{i_k,j_k})\otimes\cdots\otimes \mu_{i_m,j_m}\]
	and 
	\[\tau_k \circ \mu_{i_1,j_1,\cdots,i_m,j_m} = \mu_{i_1,j_1}\otimes\cdots\otimes(\tau\circ \mu_{i_k,j_k})\otimes\cdots\otimes \mu_{i_m,j_m}.\]
	The formula \eqref{eq:split_sigma} then follows from Lemma~\ref{lem:split}.	The direct sum decomposition \eqref{eq:split_M} follows immediately from Lemma~\ref{lem:split}.
\end{proof}
We generalize the notations $A_1$ and $\iota$. Let $A_m$ be the subgroup of $\pu{m}$ generated by $\sigma_i$ and $\tau_i$ for $1\leq i\leq m$, and let $\oin\colon A_m\to \pu{m}$ be the inclusion. Then we have the representation $\oin^*(\Psi)$ of $A_m$, which is the restriction of $\Psi$ on $A_m$. 
\begin{prop}\label{pro:V2m}
	The subgroup $A_m$ of $\pu{m}$ is isomorphic to $(\Z/(p))^{2m}$, generated by $\sigma_i$ and $\tau_i$ for $1\leq i\leq m$.
\end{prop}
\begin{proof}
	First, we show that $A_m$ is an abelian group. 
	By the definitions \eqref{eq:def_sigma} and \eqref{eq:def_tau}, for $i\neq j$, we have $\sigma_i\sigma_j = \sigma_j\sigma_i$, $\tau_i\tau_j = \tau_j\tau_i$, $\sigma_i\tau_j = \tau_j\sigma_i$. A straightforward computation shows $\td{\sigma}\td{\tau} = \omega\td{\tau}\td{\sigma}$, and so we have $\sigma_i\tau_i = \tau_i\sigma_i$. Therefore, $A_m$ is an abelian group generated by $\sigma_i$, $\tau_i$ for $1\leq i\leq m$.

	A direct computations shows  $\sigma_i^p = \tau_i^p = 1$. Therefore, $A_m$ is isomorphic to $(\Z/(p))^t$ for some $t\leq 2m$. It suffices to find a surjective homomorphism $\phi\colon A_m\to (\Z/(p))^{2m}$. 
	Let
	\[\nu_k = \mu_{0,0,\cdots,1,0,\cdots,0,0} = \mu_{0,0}\otimes\cdots\otimes\underbrace{\mu_{1,0}}_{\text{$k$th tuple}}\otimes\cdots\otimes \mu_{0,0},\]
	\[\nu_k' = \mu_{0,0,\cdots,0,1,\cdots,0,0} = \mu_{0,0}\otimes\cdots\otimes\underbrace{\mu_{0,1}}_{\text{$k$th tuple}}\otimes\cdots \otimes \mu_{0,0}.\]
	We identify $\Z/(p)$ as the multiplicative group of complex $p$th roots of unity. By Proposition~\ref{pro:split}, $\nu_k$ and $\nu_k'$ are root vectors of $\oin^*(\Psi)$,  and the corresponding roots of $\oin^*(\Psi)$ take only $p$th roots of unity as values. Therefore, for $\al\in A_m$, we have $p$th roots of unity $\omega_k,\omega_k'$, $1\leq k\leq m$, such that 
	\begin{equation*}
		\al\circ\nu_k = \omega_k \nu_k,\ \al\circ\nu_k' = \omega_k' \nu_k'.
	\end{equation*}
	We define
	\begin{equation*}
		\phi(\al) = (\omega_1,\omega_1', \cdots,\omega_m,\omega_m').
	\end{equation*}
	By Proposition~\ref{pro:split}, $\phi(\sigma_i)$ and $\phi(\tau_i)$ generate $(\Z/(p))^{2m}$, and the proof is completed.
\end{proof}

By the Weyl group of $A_m$, we mean the quotient group $N/A_m$, where $N$ is the normalizer of $A_m$ in $\pu{m}$. The conjugation action of $N$ on $A_m$ passes to an action of $N/A_m$ on $A_m$, since $A_m$ is abelian. The following proposition is the special case $n = p^m$ of Theorem~8.5 of \cite{andersen2008classification}.
\begin{prop}\label{pro:VmaxAm}
	The homomorphism  $\oin: A_m\hookrightarrow \pu{m}$ is an inclusion of a maximal nontoral elementary abelian $p$-subgroup. Up to conjugacy, this is the unique maximal nontoral elementary abelian $p$-subgroup of $\pu{m}$.

	The Weyl group of $A_m$ is isomorphic to $Sp_{2m}(\F_p)$. Regarding $A_m\cong (\Z/(p))^{2m}$ as a $2m$-dimensional vector space over $\F/p$, the group $Sp_{2m}(\F_p)$ acts on $A_m$ by left multiplication of matrices. This is the conjugation action on $A_m$ of $Sp_{2m}(\F_p)$.
\end{prop}
\begin{remark}\label{rem:Vmax}
	In Theorem 3.1 of \cite{griess1991elementary}, the maximal non-toral elementary abelian $p$-group of $SL(n,\C)$ was identified, which implies the case for $PU(n)$. In Theorem~8.5 of \cite{andersen2008classification}, the non-toral elementary abelian $p$-group $A_m$ of $PU(n)$ and its Weyl group are identified, but without the explicit description given in this section.
\end{remark}

% \begin{remark}\label{rem:Am}
% 	Proposition~\ref{pro:V2m} essentially follows from \cite[(3.1) THEOREM]{griess1991elementary} and \cite[Theorem 8.5]{andersen2008classification}. However, the reader may benefit from the presence of the concrete generates $\sigma_i$ and $\tau_i$ of $A_m$, which are missing in the aforementioned references.  	
% \end{remark}

%%%%%%%%%%%%%%%%%%%%%%%%%%%%%%%%%%%%%%%%%%%%%%%%%%%%%%%%%%%%%%%%%%%%%%%%%%%%%%%%%%%%%%%%%%%%%%%%%%%%%%%%%%%%%%

\section{The Chern classes of the conjugation representations of $A_m$}\label{sec:Chern}
In this section, we study the Chern classes of the conjugation representation of $\Psi$. 

In Section~\ref{sec:p-subgroup}, we showed  $A_m\cong(\Z/(p))^{2m}$ and that it is generated by $\sigma_i$, $\tau_i$ for $1\leq i\leq m$. We consider the cohomology of $BA_m$. The underlying additive group of the ring $\F_p$ is $\Z/(p)$, and we have the isomorphism of abelian groups
\[H^1(BA_m;\F_p)\cong \opn{Hom}(A_m,\Z/(p)).\]
Via this identification, we defined $a_i, b_i\in H^1(BA_m;\F_p)$ by 
\[a_i(\sigma_j) = b_i(\tau_j)= \delta_{ij},\ a_i(\tau_j) = b_i(\sigma_j) = 0.\]

Let $\beta$ be the Bockstein homomorphism and let $\xi_i = \beta(a_i)$, $\eta_i = \beta(b_i)$. Then $H^*(BA_m;\F_p)$ is the free graded commutative algebra over $\F_p$ generated by $a_i$, $b_i$, $\xi_i$ and $\eta_i$, for $1\leq i\leq m$. 

Let $P(A_m)$ be the subalgebra of $H^*(BA;\F_p)$ generated by $\xi_i$, $\eta_i$, $1\leq i\leq m$, which is a polynomial algebra over $\F_p$. Let $Q(A_m)$ be the subalgebra generated by $a_i, b_i$, $1\leq i\leq m$, which is an exterior algebra over $\F_p$.

The conjugation action of the Weyl group of $A_m$ induces an action of itself on $H^*(BA_m;\F_p)$. An explicit description of this action follows from Proposition~\ref{pro:VmaxAm}:
\begin{prop}\label{pro:Vmax}
	The conjugation action of the Weyl group $Sp_{2m}(\F_p)$ of $A_m$ induces an action on
	\[H^*(BA_m;\F_p) = P(A_m)Q(A_m)\]
	such that 
	\begin{itemize}
		\item it acts on $P(A_m)$ as described in Section~\ref{sec:Dickson} (with $x_i$ and $y_i$ replaced by $\xi_i$ and $\eta_i$), and
		\item it is compatible with products and the Bockstein homomorphism $\beta$. 
	\end{itemize}
\end{prop}
We consider $\gamma_{i}$, the $i$th Chern class of $\Psi$.  
\begin{prop}\label{pro:ChernClasses}
	We have
		\begin{equation*}
			\oin^*(\gamma_i) = 
			\begin{cases}
				(-1)^k C_k(\xi_1,\eta_1,\xi_2,\eta_2\cdots,\xi_m,\eta_m), \ i = p^{2m} - p^k\textrm{ for some $0\leq k\leq 2m$},\\
				0,\ \textrm{otherwise}.
			\end{cases}
		\end{equation*}
\end{prop}

\begin{proof}
	By Proposition~\ref{pro:split}, for each $(i_1,j_1\cdots,j_1,j_m)\in\mathbb{I}$, we have the representation $M_{i_1,j_1,\cdots,i_m,j_m}$ of $A_m$, of which the total Chern class is 
		\begin{equation*}
			c(M_{i_1,j_1,\cdots,i_m,j_m}) = 1 + \sum_{k=1}^m (i_k\xi_k + j_k\eta_k). 
		\end{equation*}
	Therefore, the splitting in Proposition \ref{pro:split} yields
		\begin{equation*}
			\prod_{(i_1,j_1,\cdots,i_m,j_m)\in\mathbb{I}}[1 + \sum_{k=1}^m (i_k\xi_k + j_k\eta_k)]
			 = \sum_{i=0}^{p^{2m}}\oin^*(\gamma_i).
		\end{equation*}
	The desired expression for $\oin^*(\gamma_i)$ then follows from Theorem \ref{thm:Dickson}.		
\end{proof}
For the following corollary, recall the polynomial $E_{n,i}$ defined in \eqref{eq:def_E}.
\begin{cor}\label{cor:topGamma}
	We have
		\begin{equation*}
			\oin^*(\gamma_{p^{2m}-1}) = E_{2m,2m}(\xi_1,\eta_1, \cdots,\xi_m,\eta_m)^{p-1}.
		\end{equation*}
\end{cor}
\begin{proof}
	Since $P(A_m)$ is a polynomial algebra over $\F_p$, we take $q = p$ in the equation \eqref{eq:def_E}. 
	A routine computation yields the relation
		\begin{equation*}
			E_{2m,0}(\xi_1,\eta_1, \cdots,\xi_m,\eta_m) = E_{2m,2m}(\xi_1,\eta_1, \cdots,\xi_m,\eta_m)^p.
		\end{equation*}
	The corollary then follows from (b) of Theorem~\ref{thm:Dickson}, which asserts
	\[E_{2m,0}(\xi_1,\eta_1, \cdots,\xi_m,\eta_m) = E_{2m,2m}(\xi_1,\eta_1, \cdots,\xi_m,\eta_m)C_{2m,0}(\xi_1,\eta_1, \cdots,\xi_m,\eta_m).\]
\end{proof}

\section{The classes $\al_i$}\label{sec:ImOftheta}
The group $PU(n)$ fits in a short exact sequence of Lie groups
	\begin{equation}\label{eq:PU_seq}
		1\to S^1\to U(n)\xrightarrow{f} PU(n)\to 1
	\end{equation}
which induces a homotopy fiber sequence
	\begin{equation*}
		BS^1\to BU(n)\xrightarrow{f} BPU(n).
	\end{equation*}
By \cite[Lemma 4.70]{hatcher2002algebraic}, since $BS^1$ is of the homotopy type $K(\Z,2)$ and since $BU(n)$ is simply connected, this homotopy fiber sequence is obtained by truncating the Puppe sequence of a homotopy fiber sequence of the form
\begin{equation}\label{eq:fiberseq}
	BU(n)\xrightarrow{f} BPU(n)\xrightarrow{\br}K(\Z,3).
\end{equation}
We regard the map $\br$ as a member of $H^3(BPU(n);\Z)$. Indeed, $\br$ is a generator of $H^3(BPU(n);\Z)$.

Let $n = p^m$. Consider the subgroup $\td{A}_m = f^{-1}(A_m)\cap SU(p^m)$ of $U(p^m)$. A routine computation shows that $\td{A}_m$ is the group generated by 
\begin{equation*}
	\omega, \td{\sigma}_i,\td{\tau}_i,\ 1\leq i\leq m
\end{equation*}
subjected to the relations
\begin{equation*}
	\omega^p = \td{\sigma}_i^p = \td{\tau}_i^p = 1,\  \omega\td{\sigma}_i = \td{\sigma}_i\omega,\ \omega\td{\tau}_i = \td{\tau}_i\omega,
\end{equation*}
and
\begin{equation*}
	\td{\sigma}_i\td{\tau}_j =
	\begin{cases}
		\td{\tau}_j\td{\sigma}_i,\ i\neq j,\\
		\omega\td{\tau}_i\td{\sigma}_i,\ i = j.
	\end{cases}
\end{equation*}
Indeed, $\td{A}_m$ is isomorphic to the extraspecial $p$-group $p_+^{1+2m}$.

We recall the notation for the cohomology of $BA_m$ in Proposition~\ref{pro:Vmax}:
\[H^*(BA_m;\F_p) = P(A_m)Q(A_m),\]
where 
\[P(A_m) = \F_p[\xi_1,\eta_1,\cdots,\xi_m,\eta_m],\ Q(A_m) = \Lambda_{\F_p}[a_1,b_1,\cdots,a_m,b_m].\]
Consider the central extension
\begin{equation}\label{eq:central_ext}
	1\to\Z/(p)\to\td{A}_m\xrightarrow{g}A_m\to 1.
\end{equation}

It is well known (\cite[I, Theorem 6.8]{adem2013cohomology})
that a central extension of $A_m$ with center $\Z/(p)$ 
corresponds to an element $\eg\in H^2(BA_m;\F_p)$ called the extension class of $g$. A routine computation with the bar complex of $A_m$ shows  $\eg = \sum_{i=1}^mb_ia_i$. 

The central extension \eqref{eq:central_ext} is associated to a homotopy fiber sequence 
\[B\Z/(p)\to B\td{A}_m\xrightarrow{g} BA_m,\]
which is obtained by truncating the Puppe sequence of the following homotopy fiber sequence: 
\begin{equation}\label{eq:fiberseqA}
	B\td{A}_m\xrightarrow{g} BA_m\xrightarrow{\eg}K(\Z/(p),2).
\end{equation}
The exact sequences \eqref{eq:PU_seq} and \eqref{eq:central_ext} form a commutative diagram
\begin{equation*}
	\begin{tikzcd}
		\Z/(p)\arrow[r]\arrow[d]&\td{A}_m\arrow[r,"g"]\arrow[d,"\td\iota"]&A_m\arrow[d,"\iota"]\\
		S^1\arrow[r]&U(p^m)\arrow[r,"f"]&\pu{m}
	\end{tikzcd}	
\end{equation*}
which yields a commutative diagram
\begin{equation}\label{eq:htpy_fib_diag}
	\begin{tikzcd}
		B\td{A}_m\arrow[r,"g"]\arrow[d,"\td{\oin}"]& BA_m\arrow[d,"\oin"]\arrow[r,"\eg"]& K(\Z/(p),2)\arrow[d,"\td\beta(\opn{id})"]\\
		BU(p^m)\arrow[r,"f"]& B\pu{m}\arrow[r,"\br"]& K(\Z,3)
	\end{tikzcd}
\end{equation}
where $\tilde{\beta}$ denotes the connecting homomorphism $H^*(-;\F_p)\to H^{*+1}(-;\Z)$, and $\opn{id}$ denotes the generator of $H^2(K(\Z/(p),2);\F_p)$ represented by the identity map of $K(\Z/(p),2)$.

Let $\bar{\br}\in H^3(B\pu{m};\F_p)$ be the mod $p$ reduction of $\br$. The diagram \eqref{eq:htpy_fib_diag} yields $\iota^*(\br) = \td\beta(\eg)$, and then
\begin{equation}\label{eq:thetax}
	\oin^*(\bar{\br}) = \beta(\eg) = \sum_{i=1}^m(-b_i\xi_i + a_i\eta_i)
\end{equation}
where $\beta$ denotes the Bockstein homomorphism in the mod $p$ Steenrod algebra. In \cite{gu2019some}, the author considers the integral cohomology classes 
	\begin{equation*}
		y_{p,k} =\tilde{\beta}\op^{p^k}\op^{p^{k-1}}\cdots\op^1(\bar\br)\in H^{2(p^{k+1}+1)}(B\pu{m};\Z)
	\end{equation*}
where $\op^i$ denotes the $i$th Steenrod reduced power operation. For $i>0$, let 
\begin{equation}\label{eq:Steenrod_alpha}
	\al_i = \beta\op^{p^{i-1}}\op^{p^{i-2}}\cdots\op^1(\bar{\br})\in H^{2(p^{i}+1)}(B\pu{m};\F_p).
\end{equation}
Then $\al_i$ is the mod $p$ reduction of $y_{p,i-1}$. Recall the polynomials $H_{2m,i}$ defined by \eqref{eq:def_H}. Here we have $K = \F_q = \F_p$. Let 
\[h_i =  H_{2m,i}(\xi_1,\eta_1,\cdots,\xi_m,\eta_m) = \sum_{j=1}^m(\xi_j\eta_j^{p^i} - \xi_j^{p^i}\eta_j).\]
A routine computation yields	
\begin{prop}\label{pro:alpha}
	For $i \geq 1$, we have $\oin^*(\alpha_i) = h_i$. \qed
\end{prop}
\begin{cor}\label{cor:alphanontrivial}
	We have $\al_i\neq 0$ for all $i\geq 1$.\qed
\end{cor}
By Lemma~\ref{lem:algIndependenceh} we have 
\begin{cor}\label{cor:algIndependence}
	The elements $\alpha_i$, $1\leq i\leq 2m$ are algebraically independent in $H^*(B\pu{m};\F_p)$.\qed
\end{cor}

% The following is a partial analog of Theorem 1.1 of \cite{gu2020distinguished}.
% \begin{prop}\label{pro:algIndependence}
% 	The elements $\alpha_i$ for $1\leq i\leq 2m$ are algebraically independent in $H^*(B\pu{m};\F_p)$.
% \end{prop}

% \begin{proof}
% 	The argument is almost identical to the one in \cite{gu2020distinguished}. By \eqref{eq:ritheta}, it suffices to show that $r_i$ for $1\leq i\leq 2m$ are algebraically independent in $H^*(BA_m;\F_p)$, which follows from \eqref{eq:ri}, by applying the Jacobian criterion.	
% \end{proof}

% Vistoli \cite{vistoli2007cohomology} considers the case $m=1$ and shows that $r_1$ plays a critical role in understanding the cohomology of $BPU(p)$. In particular, Vistoli shows (Lemma 5.6 of \cite{vistoli2007cohomology})
% 	\begin{equation}\label{eq:Vistoli}
% 		\begin{split}
% 			& \oin^*(A_{p^2-p}) = -\xi^{p^2-p} - \eta^{p-1}(\xi^{p-1} - \eta^{p-1})^{p-1}, \\
% 			& \oin^*(A_{p^2-1}) = r_1^{p-1} = (\xi^p\eta - \xi\eta^p)^{p-1}.
% 		\end{split}		
% 	\end{equation} 
% \begin{remark}\label{rem:alphaGamma}
Recall the subalgebra $\sG$ of $\hbpu$ defined in Definition~\ref{def:subalg_G}.
\begin{prop}\label{pro:spInv}
	In $\hbpu$, we have 
	\[\oin^*(\sG) = \opn{Im}\oin^*\cap P(A_m) = P(A_m)^{Sp_{2m}(\F_p)}.\]
\end{prop}
\begin{proof}
	The inclusion $\oin^*(\sG)\subseteq\opn{Im}\oin^*\cap P(A_m)$ is clear. Since elements in $\opn{Im}\oin^*$ are invariant under the action of the Weyl group of $A_m$, the inclusion 
	\[\opn{Im}\oin^*\cap P(A_m) \subseteq P(A_m)^{Sp_{2m}(\F_p)}\] follows from Proposition~\ref{pro:Vmax}. The inclusion $P(A_m)^{Sp_{2m}(\F_p)}\subseteq\oin^*(\sG)$ follows from Corollary~\ref{cor:sympInv}.
\end{proof}

For the rest of this section, let $P =P(A_m)$, $Q = Q(A_m)$ and let $Q^+$ be the $\F_p$-submodule of $Q$ generated by elements of positive dimensions. We have 
\begin{equation}\label{eq:PQ}
	H^*(BA_m;\F_p) = PQ = P\oplus PQ^+.
\end{equation}
Therefore, we have an isomorphism
\begin{equation}\label{eq:PQnilradical}
	P\hookrightarrow H^*(BA_m;\F_p)\twoheadrightarrow H^*(BA_m;\F_p)/PQ^+
\end{equation}
where the arrows are the obvious ones. Notice that $PQ^+$ is the nilradical of $H^*(BA_m;\F_p)$, which implies that $\oin^*\colon\hbpu\to H^*(BA_m;\F_p)$ induces a homomorphism
\[\oin^*_{\sN}\colon\hbpu/\sN\to P.\]
Proposition~\ref{pro:spInv} implies the following
\begin{cor}\label{cor:spInv}
	$\opn{Im}\oin^*_{\sN} = P(A_m)^{Sp_{2m}(\F_p)}$.\qed
\end{cor}

%%%%%%%%%%%%%%%%%%%%%%%%%%%%%%%%%%%%%%%%%%%%%%%%%%%%%%%%%%%%%%%%%%%%%%%%%%%%%%%%%%%%%%%%%%%%%
\section{Some consequences of Quillen's Theorem}\label{sec:Quillen}
We prove Theorem~\ref{thm:ChernClasses} and Theorem~\ref{thm:polynomial}. 
\begin{lem}\label{lem:Quillen}
	Let $\kappa: T\to\pu{m}$ be an inclusion of a maximal torus. 
	Let $u\in H^*(BPU(p^m);\F_p)$ such that $\kappa^*(u) = 0$ and $\oin^*(u) = 0$. Then $u$ is nilpotent. 
\end{lem}
	
\begin{proof}
	We recall the work of Quillen \cite{quillen1971spectrum}. Let $G$ be a compact Lie group, and  $\mathscr{A}$ the category of elementary abelian $p$-subgroups of $G$ and inclusions, and define the functor
	\begin{equation*}
		\mathscr{H}: \mathscr{A}^{op}\to \opn{Vect}\F_p,\ A\mapsto H^*(BA;\F_p)
	\end{equation*}
	where $\opn{Vect}\F_p$ denotes the category of $\F_p$-vector spaces. We have a canonical homomorphism $\Theta \colon H^*(BG;\F_p)\to\opn{lim}\mathscr{H}$.
	
	Quillen \cite{quillen1971spectrum} shows 
	\begin{equation}\label{eq:tautological}
		\opn{Ker}\Theta\subseteq\sN.
	\end{equation}
	It follows from Proposition \ref{pro:Vmax} that, up to conjugacy,  $\oin: A_m\to\pu{m}$ is the inclusion of the unique maximal non-toral elementary abelian $p$-group. The lemma then follows from \eqref{eq:tautological}. 
\end{proof}

\begin{thmrep}[\ref{thm:ChernClasses}]
	The Chern class $\gamma_i$ is  nilpotent if
	\begin{enumerate}[label = (\alph*)]
		\item\label{ite:Chern1} $i$ is odd, or
		\item \label{ite:Chern2} $i > p^{2m} - p^m$ and $i\neq p^{2m} - p^k$ for any positive integer $k$.
	\end{enumerate} 
\end{thmrep}

\begin{proof}
	In this proof, we used the terminology ``Chern classes'' to refer to the mod $p$ reduction of the conventional Chern classes, which are integral cohomology classes by definition.

	Let $\kappa: T\to\pu{m}$ be as in Lemma~\ref{lem:Quillen}. Let $f\colon U(p^m)\to\pu{m}$ be the quotient map. Let $\td{T} = f^{-1}(T)$. Then $\td{T}$ is a maximal torus of $U(p^m)$, and 
	\[H^*(B\td{T};\F_p) = \F_p[t_1,\cdots,t_{p^m}]\]
	where $t_i$ is the first Chern class of a $1$-dimensional direct summand of the universal vector bundle on $B\td{T}$. A routine computation shows that 
	\[f^*\colon H^*(BT;\F_p)\to H^*(B\td{T};\F_p)\]
	is an injection with image the sub-algebra generated by $1$ and $t_i - t_j$, for $1\leq i,j\leq p^m$. Therefore, we identify elements of $H^*(BT;\F_p)$ with their images in $H^*(B\td{T};\F_p)$ via $f^*$.

	Let $\epsilon_{ij}\in M(p^m)$ be the matrix of which the entry on the $i$th row and $j$th column is $1$, whereas the other entries are $0$. 
	The restriction of $\kappa^*(\Psi)$ splits into a direct sum of $1$-dimensional representations of the form $\C\epsilon_{ij}$, of which the first Chern class is $t_i - t_j$. Therefore, the total Chern class of $\kappa^*(\Psi)$ is 
	\begin{equation}\label{eq:torusChernClass}
		\begin{split}
			c(\kappa^*(\Psi)) & = \prod_{1\leq i,j\leq p^m}[1 + (t_i - t_j)] = \prod_{1\leq i < j\leq p^m}[1 - (t_i - t_j)^2].
		\end{split}	
	\end{equation}
	Therefore, $\kappa^*(c_i(\Psi)) = c_i(\kappa^*(\Psi)) = 0$ if
	\begin{enumerate}
		\item $i$ is odd, or
		\item $i > p^{2m} - p^m$.	
	\end{enumerate} 
	Comparing this with Proposition~\ref{pro:ChernClasses}, it follows that $\oin^*(c_i(\Psi)) = 0$ and $\kappa^*(c_i(\Psi)) = 0$ if \ref{ite:Chern1} or \ref{ite:Chern2} holds. The desired conclusion then follows from Lemma~\ref{lem:Quillen}.	
\end{proof}

The equation \eqref{eq:torusChernClass} holds for integral cohomology as well, from which we deduce
\begin{cor}\label{cor:torsionChernClass}
	In the integral cohomology, $c_i(\Psi)$ is torsion if 
	\begin{itemize}
		\item $i$ is odd, or
		\item $i > p^{2m} - p^m$.
	\end{itemize}
	\qed
\end{cor}
\begin{lem}\label{lem:kappaAlpha}
	$\kappa^*(\al_i) = 0$.
\end{lem}
\begin{proof}
	This follows from the facts that $\al_i$'s are mod $p$ reductions of torsion classes and that $H^*(BT;\Z)$ is torsion-free.
\end{proof}
\begin{thmrep}[\ref{thm:polynomial}]
	The following relations hold in $\hbpu$.
	\begin{equation}\label{eq:relation1'}
		\sum_j^{i-1}\al_{i-j}^{p^j}\gmm{j} \equiv \sum_{j=i+1}^{2m}\al_{j-i}^{p^i}\gmm{j} \pmod{\sN}, \ 0\leq i\leq m-1,		
	\end{equation}

	\begin{equation}\label{eq:relation2'}
		\begin{split}
			\gmm{i} \equiv
			\sum_{j=m-i}^m (-1)^{m+i+j}P_{m-i,j}(\al_1,\cdots,\al_{2j-1})^{p^{m-j}}&\gmm{m+j}
			\pmod{\sN}, \\
			& 0\leq i\leq m-1,
		\end{split}
	\end{equation}

	\begin{equation}\label{eq:relation4'}
		\al_i \equiv \sum_{j=1}^{2m-1}(-1)^{j+1} R_{i,j}(\gmm{0},\cdots,\gmm{2m-1})\al_j\pmod{\sN},\ i\geq 2m,
	\end{equation}

	\begin{equation}\label{eq:relation3'}
		D_i(\al_1,\cdots,\al_{2m})\equiv \gmm{i}  D_{2m}(\al_1\cdots,\al_{2m-1}) \pmod{\mathscr{N}},\ 0\leq i\leq 2m.
	\end{equation}
\end{thmrep}
\begin{proof}
	Let $LHS$ and $RHS$ denote the two sides of the modular equivalences. To show each of the three modular equivalences, by Lemma~\ref{lem:Quillen}, it suffices to show $\oin^*(LHS) = \oin^*(RHS)$ and $\kappa^*(LHS) = \kappa^*(RHS)$. 

	We prove \eqref{eq:relation1'}. By Lemma~\ref{lem:kappaAlpha}, we have $\kappa^*(\al_i) = 0$ and then 
	\[\kappa^*(LHS) = \kappa^*(RHS) = 0.\]
	The relation $\oin^*(LHS) = \oin^*(RHS)$ follows from \eqref{eq:sympInv1} in Theorem~\ref{thm:sympInv}.
	
	We prove \eqref{eq:relation2'}. By Corollary~\ref{cor:torsionChernClass}, we have $\kappa^*(LHS) = 0$. By Lemma~\ref{lem:kappaAlpha}, we have $\kappa^*(RHS) = 0$. The relation $\oin^*(LHS) = \oin^*(RHS)$ follows from \eqref{eq:sympInv2} in Theorem~\ref{thm:sympInv}.

	We prove \eqref{eq:relation4'}. By Lemma~\ref{lem:kappaAlpha}, we have $\kappa^*(LHS)=\kappa^*(RHS) = 0$.
	By Proposition~\ref{pro:h2m}, we have $\oin^*(LHS) = \oin^*(RHS)$.

	We prove \eqref{eq:relation3'}. Again, the relation
	\[\kappa^*(LHS) = \kappa^*(RHS) = 0\]
	follows from Lemma~\ref{lem:kappaAlpha}. By Corollary~\ref{cor:CandD}, we have $\oin^*(LHS) = \oin^*(RHS)$.
\end{proof}

We conclude this section with 
\begin{prop}\label{pro:phi_alpha}
	Let $\sG$ be the $\F_p$-subalgebra of $\hbpu$ generated by $1$ and $\al_i$ for $1\leq i\leq 2m-1$ and $\gmm{i}$ for $m\leq i\leq 2m-1$. Let $\varphi(\underline{\al})$ be a polynomial in the classes $\alpha_i$ for $i > 0$ such that $\varphi(\underline{\al})\in H^k(B\pu{m};\F_p)$ for some $k > 0$. Then we have the equality of $\F_p$-submodules of $\hbpu/\sN$ 
	\[\hbpu\varphi(\underline{\al})\equiv\sG\varphi(\underline{\al})\pmod{\sN}.\]
\end{prop}
\begin{proof}
	Throughout this proof, we set  $\varphi = \varphi(\underline{\al})$, $P = P(A_m)$, $Q = Q(A_m)$, and $Q^+$ the $\F_p$-submodule of $Q$ consisting of elements of positive dimensions. Then we have 
	\[H^*(BA_m;\F_p) \cong P\otimes Q = P\oplus PQ^+.\]
	Notice that elements in $PQ^+$ are nilpotent. For $u\in\hbpu$, let $v = \oin^*(u) = v_1 + v_2$ where $v_1\in P$ and $v_2\in PQ^+$. By Theorem~\ref{thm:sympInv}, there is $u_1\in\sG$ such that $\oin^*(u_1)=v_1$. 
	
	Let $u_2 = u - u_1$. Then $\oin^*(u_2) = v_2\in PQ^+$, which implies that $\oin^*(u_2)$ is nilpotent, and so is $\oin^*(u_2\varphi)$. 

	Furthermore, since $\kappa^*(\al_i) = 0$, we have $\kappa^*(u_2\varphi) = 0$. By Lemma~\ref{lem:Quillen}, $u_2\varphi$ is nilpotent, and the proof is completed.
\end{proof}

%%%%%%%%%%%%%%%%%%%%%%%%%%%%%%%%%%%%%%%%%%%%%%%%%%%%%%%%%%%%%%%%%%%%%%%%%%%%%%%%%%%%%%%%%%%%%%%%%%%%%%%%%%%%%%
\section{The mod $p$ reduction of torsion cohomology classes}\label{sec:torsion}
In this section we prove 
\begin{thmrep}[\ref{thm:ideals}]
	Regarding $\sT$, $\sJ$ and $\sN$ as $\F_p$-submodules of $\hbpu$, we have 
	\[\sT \equiv \sJ \pmod{\sN}.\]
\end{thmrep}

 Recall the homomorphism $g\colon \td{A}_m\to A_m$ considered in Section~\ref{sec:ImOftheta}. We have the homotopy commutative diagram
 \begin{equation}\label{eq:commDiag}
	\begin{tikzcd}
		B\td{A}_m\arrow[r,"g"]\arrow[d,"\td{\oin}"]& BA_m\arrow[d,"\oin"]\\
		BU(p^m)\arrow[r,"f"]& BPU(p^m)
	\end{tikzcd}
\end{equation}

 Let $J$ be the Jacobson radical of $H^*(B\td{A}_m;\F_p)$. Consider  the composition
\[\td{g}^*\colon H^*(BA_m;\F_p)\xrightarrow{g^*} H^*(B\td{A}_m;\F_p)\to H^*(B\td{A}_m;\F_p)/J\]
where the second arrow is the quotient homomorphism. Let $I$ be the ideal of $P(A_m)$ generated by $h_i$ for $1\leq i\leq 2m-1$. By Proposition~\ref{pro:h2m}, we have $h_i\in I$ for $i\geq 1$.  
The following is an immediate consequence of Theorem~10.1 of \cite{benson1992cohomology} and the erratum \cite{benson1993erratum}.
\begin{lem}\label{lem:KerOfg}
	$\opn{Ker}\td{g}^*|_{P(A_m)} \subseteq I$.\qed
\end{lem}

\begin{prop}\label{pro:KerOfg}
	 $\opn{Ker}g^*|_{P(A_m)} = \opn{Ker}\td{g}^*|_{P(A_m)} = I$.
\end{prop}
\begin{proof}
	By Lemma~\ref{lem:KerOfg}, we have 
	$\opn{Ker}\td{g}^*|_{P(A_m)} \subseteq I$. On the other hand, 
	We have 
	\[\opn{Ker}g^*|_{P(A_m)} \subseteq \opn{Ker}\td{g}^*|_{P(A_m)}\] 
	by the definitions and $g^*$ and $\td{g}^*$.  
		
	Consider the classes $\al_i\in\hbpu$. We have $f^*(\alpha_i)=0$ since $\al_i$ are mod $p$ reductions of torsion classes. 	By \eqref{eq:commDiag}, we have 
	\[g^*(h_i) = g^*\oin^*(\al_i) = \td{\oin}^*f^*(\al_i) = 0.\]
	Therefore, we have $h_i\in\opn{Ker}g^*|_{P(A_m)}$, and then $I\subseteq\opn{Ker}g^*|_{P(A_m)}$. 
\end{proof}

\begin{lem}\label{lem:IdealI}
	In $P(A_m)$, we have
	$\oin^*(\sJ) = P(A_m)^{\spl{m}}\cap I$.
\end{lem}
\begin{proof}
	By Corollary~\ref{cor:sympInv}, the ideal of $P(A_m)^{\spl{m}}$ generated by $h_i$ for $1\leq i\leq 2m-1$ is precisely $\oin^*(\sJ)$. The lemma then follows immediately from Proposition~\ref{pro:idealH}. 
\end{proof}

% \begin{remark}\label{rem:tekuza}
% 	Let $\opn{Var}(I)$ be the algebraic variety associated with $I$. In \cite{tezuka1983varieties}, the authors verified that the subvariety characterized by $\eta_i = 0$ for $1\leq i\leq m$ is an irreducible component of $\opn{Var}(I)$.
% \end{remark}

\begin{proof}[Proof of Theorem~\ref{thm:ideals}]
	Recall \eqref{eq:PQ}:
	\[H^*(BA_m;\F_p) = PQ = P\oplus PQ^+\]
	in which we have $P =P(A_m), Q = Q(A_m)$, and $Q^+$ the $\F_p$-submodule of $Q$ generated by elements of positive dimensions.
	
	Let $u\in\sT$, $v = \oin^*(u) = v_1 + v_2$, where $v_1\in P$ and $v_2\in PQ^+$. Since $P$ and $PQ^+$ are invariant subspaces of the $\spl{m}$-action on $H^*(BA_m;\F_p)$, we have $v_1\in P^{\spl{m}}$. 
	
	Since $u\in\sT$, we have $f^*(u) = 0$, and by \eqref{eq:commDiag}, we have $v\in\opn{Ker}(g^*)\subseteq\opn{Ker}(\td{g}^*)$. Since $v_2\in PQ^+$, $v_2$ is nilpotent. Therefore we have $v_2\in \opn{Ker}(\td{g}^*)$. By Proposition~\ref{pro:KerOfg}, we have 
	\begin{equation*}
		v_1 = v - v_2\in \opn{Ker}\td{g}^*|_{P} = I.
	\end{equation*}
	By Lemma~\ref{lem:IdealI}, we have
	\begin{equation*}
		v_1\in  P^{\spl{m}}\cap I = \oin^*(\sJ).
	\end{equation*}
	Therefore, we may take $u_1\in\sJ$ such that $v_1 = \oin^*(u_1)$. Let $u_2 = u - u_1$. Since $u, u_1\in\sT$, we have
	\[\kappa^*(u_2) = \kappa^*(u) - \kappa^*(u_1) = 0.\]
	On the other hand, the class $\oin^*(u_2) = v_2$ is nilpotent. Therefore,  $\oin^*(u_2^r) = 0$ for some $r\in\mathbb{N}$. 	By Lemma~\ref{lem:Quillen}, we have $u_2^r\in\sN$, which yields $u_2\in\sN$. We have shown $\sT\subseteq \sJ + \sN$. 

	Since the classes $\al_i$ are mod $p$ reductions of integral torsion cohomology classes, we have $\sJ\subseteq\sT$. The proof is completed. 	
\end{proof}

\bibliographystyle{alpha}
\bibliography{RefConjRep}
\end{document}